\documentclass[12pt]{amsart}

\usepackage{amsmath,amssymb,amscd,amsfonts,amsthm,verbatim}
\usepackage[all,cmtip]{xy}
\usepackage{paralist}
\usepackage[mathscr]{eucal}
\usepackage{color}
\usepackage{pdfsync}
\usepackage[]{fontenc}
\usepackage{enumerate}

\usepackage[pdftex]{hyperref}
\hypersetup{citecolor=blue,linktocpage}

\usepackage[colorinlistoftodos]{todonotes}
\newtheorem{thm}{Theorem}[section]
\newtheorem{lem}[thm]{Lemma}
\newtheorem{prop}[thm]{Proposition}
\newtheorem{cor}[thm]{Corollary}

\newtheorem{assu-nota}[thm]{Assumption--Notation}

\theoremstyle{definition}
\newtheorem{defn}[thm]{Definition}
\newtheorem{rem}[thm]{Remark}
\newtheorem{ex}[thm]{Example}
\newtheorem{qst}[thm]{Question}

\newcommand{\inv}{^{-1}}

\newcommand{\C}{\mathbb C}
\newcommand{\Z}{\mathbb Z}
\newcommand{\Q}{\mathbb Q}

\newcommand{\pp}{\mathbb P}

\newcommand{\ud}{^{(d)}}
\newcommand{\up}[1]{^{(#1)}}

\newcommand{\OO}{\mathcal O}

\DeclareMathOperator{\Pic}{Pic}
\DeclareMathOperator{\pic}{Pic}

\DeclareMathOperator{\Alb}{Alb}

\DeclareMathOperator{\Hom}{Hom}

\DeclareMathOperator{\pr}{pr}

\DeclareMathOperator{\divis}{div}

\DeclareMathOperator{\vol}{vol}

\newcommand{\ol}{\overline}
\newcommand{\wt}{\widetilde}

\numberwithin{equation}{section}

\title[The eventual paracanonical map\dots]{The eventual paracanonical map of a variety of maximal Albanese dimension}
\author{Miguel \'Angel Barja,  Rita Pardini and Lidia Stoppino}
\date{}

\begin{document}

\begin{abstract}
Let $X$ be a smooth complex projective variety   such that the Albanese map of $X$ is generically finite onto its image.
Here we study  the so-called {\em eventual $m$-paracanonical map} of $X$, whose existence is implied by the results of \cite{articolone} (when $m=1$ we also assume $\chi(K_X)>0$).\par
We show that for $m=1$ this map  behaves in  a similar way  to the canonical map of a surface of general type, as described in \cite{beauville-canonique},  while it is birational for $m>1$.
We also describe it explicitly  in several examples.
\end{abstract}

\maketitle
\section{Introduction}
Let $X$ be a smooth complex projective variety    and let   $a\colon X\to A$ a map into an abelian variety that is  generically finite onto its image and such that $a^*\colon\pic^0(A)\to \pic^0(X)$ is injective; we identify $\pic^0(A)$ with a subgroup of  $\pic^0(X)$ via $a^*$.

Let $L\in \Pic(X)$ be  a line bundle  such that $h^0(L\otimes \alpha)>0$ for every $\alpha\in \pic^0(A)$. We may consider the  pull back $L\up{d}$ to the \`etale cover $X\up{d}\to X$ induced by the $d$-th multiplication map of $A$. In \cite{articolone} it is shown that for $d$ large and divisible enough and $\alpha\in\pic^0(A)$ general  the map given by $|L\up{d}\otimes \alpha|$ does not depend on $\alpha$  and is  obtained by base change from the so-called {\em eventual map} associated to $L$,  a generically finite map $\phi\colon X\to Z$, uniquely determined up to birational isomorphism, such that $a$ factors through it.

Here we consider the case when $a\colon X\to A$ is the Albanese map and $L=mK_X$, $m\ge 1$,  and study the  {\em eventual $m$-paracanonical map}, which is intrinsically attached to $A$. (For $m=1$ one needs to assume also that  $\chi(K_X)>0$, since by the generic vanishing theorem $\chi(K_X)$ is the generic value of $h^0(K_X\otimes \alpha)$).

After recalling (\S \ref{sec:prelim}) the main properties of the eventual map, in \S  \ref{sec:para} we concentrate on the case $m=1$ and prove  that the relations between the numerical invariants of $X$ and of the paracanonical image $Z$ are completely analogous to those between a surface of general type and its canonical image (see \cite{beauville-canonique}).
In  \S \ref{sec:m>1} we prove that the eventual $m$-canonical map is always birational for $m>1$, and therefore $m=1$ is the only interesting case.

The last section contains several  examples,  that show that the eventual paracanonical map is really a new object and needs not be birational or coincide   with the Albanese map, and can have arbitrarily large degree. We finish by posing a couple of questions.
\smallskip

{\em Acknowledgements:} we wish to thank Christopher Hacon for very useful  mathematical communications, and especially  for pointing out the reference \cite{HMc}. The second and third named author would like to thank the Departament de Matem\`atiques of the Universitat Polit\`ecnica de Catalunya for the invitation and the warm hospitality.
\bigskip

{\bf Conventions:}  We work over the complex numbers. \par
 ``map'' means  rational map. \par
Given maps $f\colon X\to Y$ and $g\colon X\to Z$, we say that $g$ is composed with $f$ if there exists a map $h\colon Y\to Z$ such that $g=h\circ f$. Given a map  $f\colon X\to Y$  and an involution $\sigma$ of $X$, we say that $f$ is composed with $\sigma$ if $f\circ\sigma=f$.
\par
We say that two dominant maps $f\colon X\to Z$, $f'\colon X\to Z'$ are birationally equivalent if there exists a birational isomorphism $h\colon Z\to Z'$ such that $f'=h\circ f$.
\par
A generically finite map means a rational map that is generically finite \underline{onto its image}.
A smooth projective variety $X$ is of maximal Albanese dimension if its Albanese map is generically finite.
\par The symbol $\equiv$ denotes numerical equivalence of  line bundles/$\Q$-divisors.
\section{Preliminaries}\label{sec:prelim}

In this section we fix the set-up and recall some facts and definitions. More details can be found in  \cite{barja-severi}, \cite{ugua}, \cite{articolone}.
\medskip

Let $X$ be a smooth projective map, let $a\colon X\to A$ be a generically finite map to an abelian variety of dimension $q$. Given a line bundle $L$ on $X$, the {\em continuous rank} $h^0_a(L)$  of $L$ (with respect to $a$) is defined as the  minimum value of $h^0(L\otimes\alpha)$ as $\alpha $  varies in $\Pic^0(A)$.

 Let $d$ be an integer and denote by $\mu_d\colon A\to A$ the multiplication by $d$; the following cartesian diagram defines the variety $X\up{d}$ and the maps $a_d$ and $\wt{\mu_d}$.
 \begin{equation}\label{diag:0}
\begin{CD}
X\up{d}@>\wt{\mu_d}>>X\\
@V{a_d}VV @VV a V\\
A@>{\mu_d}>>A
\end{CD}
\end{equation}
We set $L\ud:=\wt{\mu_d}^*L$; one has $h^0_a(L\ud)=d^{2q}h^0_a(L)$.\par
 We say that $a$ is  {\em strongly generating} if the induced map $\Pic^0(A)\to\Pic^0(X)$ is an inclusion; notice that   if $a$ is strongly generating, then the variety $X\up{d}$ is connected for every $d$.

\begin{defn} We say that  a certain property holds {\em generically}  for $L$ (with respect to $a$) iff it holds for $L \otimes \alpha$  for general $\alpha\in\Pic^0(A)$; similarly, we say that a property holds {\em eventually}  (with respect to $a$) for $L$ iff  it holds for $L\ud$ for $d$ sufficiently large and divisible. For instance, we say that $L$ is {\em eventually generically birational } if  for $d$ sufficiently large and divisible  $L\ud\otimes \alpha$ is birational for $\alpha\in \Pic^0(A)$ general.
\end{defn}

 In \cite{articolone}  we have studied  the eventual generic  behaviour of a line bundle $L$ with $h^0_a(L)>0$:

 \begin{thm}[Thm. A of \cite{articolone}]\label{thm:factorization}
 Let $X$ be  a smooth projective variety,  let  be a strongly generating  and generically finite map $a\colon X\to A$ to  an abelian variety $A$,
 and let $L\in \Pic(X)$ be such that $h^0_a(L)>0$.

 Then there exists a generically finite map $\varphi\colon X\to Z$,  unique up to birational equivalence, such that:
\begin{itemize}
\item[(a)] the map $a\colon X \to A$ is composed with $\varphi$.
\item[(b)] for  $d\ge 1$ denote by $\varphi\up{d}\colon X\up{d}\to Z\up{d}$ the map obtained  from  $\varphi\colon X\to Z$ by   taking base change with $\mu_d$;  then
 the map given by $|L\ud\otimes \alpha|$  is composed with $\varphi\up{d}$ for $\alpha\in \Pic^0(A)$ general.
\item[(c)] for $d$ sufficiently large and divisible, the map $\varphi\up{d}$ is birationally equivalent to the map given by $|L\ud\otimes \alpha|$ for $\alpha\in \Pic^0(A)$ general.
\end{itemize}
 \end{thm}

 The map $\varphi\colon X\to Z$  is called the {\em eventual map} given by $L$ and the degree $m_L$ of $\varphi$ is called the {\em eventual degree} of $L$.

\section{Eventual  behaviour of the paracanonical system}\label{sec:para}

With the same notation as in \S \ref{sec:prelim}, we assume  that $X$ is a smooth projective variety of maximal Albanese dimension  with $\chi(K_X)>0$ and we consider  $L=K_X$. Since $h^0_a(K_X)=\chi(K_X)$ by generic vanishing, Theorem \ref{thm:factorization} can be applied to   $L=K_X$ taking as  $a\colon X\to A$  the Albanese map of $X$.

We call the eventual map  $\varphi\colon X\to Z$ given by $K_X$ the {\em eventual paracanonical map}, its image $Z$  the {\em eventual paracanonical image} and its  degree $m_X:=m_{K_X}$ the  {\em eventual paracanonical degree}.

%Since $Z$ is determined only up to birational isomorphism, we may replace it by one of its  smooth models  $\wt Z$ smooth;  then note  that the universal property of the Albanese map implies that the induced map $\beta \colon \wt Z \to  A$ is the Albanese map of $\wt Z$.

The following result is analogous to Thm.~3.1 of \cite{beauville-canonique}.
\begin{thm} \label{thm:chi}
In the above set-up, denote by $\wt{Z}$ any smooth model  of $Z$. Then one of the following occurs:
\begin{itemize}
\item[(a)] $\chi(K_{\wt{Z}})=\chi(K_{X})$
\item[(b)] $\chi(K_{\wt{Z}})=0$.
\end{itemize}
Furthermore, in case (b) the Albanese image of $X$ is ruled by tori.
\end{thm}

\begin{proof}
We adapt  the proof of Thm.~3.1 of \cite{beauville-canonique}.

Assume that $\chi(K_{\wt{Z}})>0$. Up to replacing  $X$ by a suitable   smooth birational model, we may assume  that the induced map $\varphi\colon X\to\wt{Z}$ is a morphism. Recall that by Theorem \ref{thm:factorization}, the Albanese map of $X$ is composed with $\varphi$.
%; we denote by $\beta\colon \wt Z \to A$ the induced map.  Given $\alpha\in \Pic^0(A)$ we denote again by $\alpha$ its pull-back to $\wt{Z}$ and to $X$.
 Fix  $\alpha\in \Pic^0(A)$ general, so that in particular  the map given by $|K_{X}\otimes \alpha|$ is composed with $\varphi$.

 Pick   $0\ne\omega \in H^0( K_{\wt{Z}}\otimes \alpha)$, denote by $\varphi^*\omega\in H^0(K_X\otimes \alpha)$ the pull-back of $\omega$ regarded as a differential form with values in $\alpha$, and  set $D:=\divis(\omega)$.  The Hurwitz formula gives
$\divis(\varphi^*\omega)=\varphi^*D+R$, where $R$ is the ramification divisor of $\varphi$. On the other hand, by assumption we have that  $\divis(\varphi^*\omega)=\varphi^*H_0+F$, where $H_0$ is an effective divisor of $\wt{Z}$ and $F$ is the fixed part of $|K_{X}\otimes \alpha|$. Hence we have:
\begin{equation}\label{eq:pullback}
\varphi^*D+R=\varphi^*H_0+F.
\end{equation}
We want to use the above relation to show that $H_0\le D$.
Let $\Gamma$ be a prime divisor of $\wt{Z}$, let $\Gamma'$ be a component of $\varphi^*\Gamma$ and let $e$ be the ramification index of $\varphi$ along $\Gamma'$, that is, $e$ is the multiplicity of $\Gamma'$ in $\varphi^*\Gamma$. Note that $\Gamma'$ appears in $R$ with multiplicity $e-1$.
If  $a$ is the multiplicity of $\Gamma$ in $D$ and $b$ is   the multiplicity of $\Gamma$ in $H_0$, then  comparing  the multiplicity of $\Gamma'$  in both sides  of \eqref{eq:pullback} we  get $ea+e-1\ge eb$, namely $a\ge b$.
This shows that $H_0\le D$ and we have the following chain of inequalities:
\begin{equation} \label{eq:eq}
h^0(H_0)\le h^0(K_{\wt{Z}}\otimes\alpha)\le h^0(K_{X}\otimes\alpha)=h^0(H_0).
\end{equation}
So all the inequalities  in \eqref{eq:eq} are actually equalities and, by the generality of $\alpha$, we have $\chi(K_{\wt{Z}})= h^0(K_{\wt{Z}}\otimes\alpha)= h^0(K_{X}\otimes\alpha)=\chi(K_{X})$\smallskip

Assume  $\chi(K_{\wt{Z}})=0$, instead. Then the Albanese image of $\wt Z$ is ruled by tori by \cite[Thm.~3]{ein-lazarsfeld}. Let $\beta\colon \wt Z\to A$ be the map such that $a=\beta \circ \varphi$. By the universal property of the Albanese map, we have that $\beta$
%On the other hand, the universal property of the Albanese map implies that the map $\beta\colon \wt Z\to A$ induced by $\varphi$ 
 is the Albanese map of $\wt Z$, hence $a(X)=\beta(\wt Z)$ is also ruled by tori.
\end{proof}

\begin{cor}\label{cor:severi}
In the above set-up, denote by $n$ the dimension of $X$.  Then:
 \begin{enumerate}
\item $\vol(K_X)\ge m_X n! \chi(K_X)$
\item if $\chi(K_{\wt Z})>0$, then  $\vol(K_X)\ge 2m_X n! \chi(K_X)$
\end{enumerate}
In particular, the stronger inequality (ii) holds whenever the Albanese image of $X$ is not ruled by tori.
\end{cor}
\begin{proof}
(i)  Follows directly by Corollary 3.12  of \cite{articolone}.
\smallskip

For    (ii) observe that if $\chi(K_{\wt Z})>0$, then $\chi(K_X)=\chi(K_{\wt Z})$ by Theorem  \ref{thm:chi};  the Main Theorem of \cite{barja-severi} gives $\vol(K_{\wt Z})\ge 2n!\chi(K_{\wt Z})=2n!\chi(K_X)$ and, arguing as in Corollary 3.12, we have inequality (ii).

Finally, the last sentence in the statement is a consequence of Theorem \ref{thm:chi}.
\end{proof}

One can be more precise in the surface case:
\begin{prop}\label{prop:surface}
In the above set-up, assume that   $\dim X=2$ and $X$ is minimal of general type.
Then one of the following cases  occurs:
\begin{enumerate}
\item[(a)]  $p_g(X)=p_g(\wt Z)$, $q(X)=q(\wt Z)$ and    $m_X\le 2$
\item[(b)]  $\chi(\wt Z)=0$,   $2\le m_X\le 4$ and the Albanese image of $X$ is ruled by tori.
\end{enumerate}
\end{prop}
\begin{proof}
Assume $\chi(\wt Z)>0$; then by Theorem  \ref{thm:chi} we have  $\chi(\wt Z)=\chi(X)$ and Corollary \ref{cor:severi} (ii) gives  $K^2_X\ge 4m_X\chi(X)$. Hence  $m_X\le 2$
follows immediately by the Bogomolov-Miyaoka-Yau inequality $K^2_X\le 9\chi(K_X)$. We have $q(Z)=q(X)$,  since the Albanese map of $X$ is composed with $\varphi\colon X\to Z$, so $p_g(X)=\chi(X)+q(X)-1=\chi(\wt Z)+q(\wt Z)-1=p_g(\wt Z)$.
\smallskip

Assume now that $\chi(\wt Z)=0$; in this case the Albanese image of $X$ is ruled by tori  by Theorem \ref{thm:chi} and Corollary \ref{cor:severi} (i) gives  $K^2_X\ge 2m_X\chi(X)$. Using the Bogomolov-Miyaoka-Yau inequality as above we get $m_X\le 4$. On the other hand $m_X>1$, since $\varphi$ cannot be birational because $\chi(X)>\chi(\wt Z)$.
\end{proof}

\begin{rem}
For the case of a minimal smooth threefold of general type $X$, it is proven in \cite{CCZ} that the inequality $K_X^3\leq 72 \chi(X)$ holds. Hence the   arguments used to prove  Proposition \ref{prop:surface} above show that in this case  we have $m_X\leq 6$ if $\chi(\wt Z)>0$,  and  $m_X\leq 12$ if $\chi(\wt Z)=0$.
\end{rem}

 \section{Eventual behaviour of the  $m$-paracanonical system for $m>1$}\label{sec:m>1}

 In this section we consider  the eventual $m$-paracanonical map for varieties of general type and maximal Albanese dimension, for $m\ge 2$. Our main result implies
that   for $m\ge 2$ the eventual $m$-paracanonical map does not give additional information on the geometry of $X$:
 \begin{thm}\label{thm:pluri-para}
 Let $X$ be a smooth projective $n$-dimensional variety of general type and maximal Albanese dimension; denote by $a\colon X\to A$  the Albanese map.

  Then there exists  a positive integer $d$ such that the  system $|mK_{X\up{d}}\otimes \alpha|$ is   birational for every $m\ge 2$ and for every $\alpha\in \Pic^0(A)$.

  In particular $mK_X$ is eventually generically  birational for $m\ge 2$.
 \end{thm}

 \begin{rem}
% For $m\ge 2$,  one has $h^0(mK_X\otimes \alpha)=h^0_a(mK_X)$  for every $\alpha\in \Pic^0(X)$, hence if the system $|mK_X|$ is birational then $|mK_X\otimes \alpha|$ is birational for $\alpha$ in an open set $U\subset \Pic^0(X)$.
 The question whether the $m$-canonical system of  a variety of maximal Albanese dimension is birational has been considered by several authors.

The answer is positive for   $m\ge 3$ (\cite{JLT},  \cite{chen-hacon}, \cite{pareschi-popa}).

 The case $m=2$ has been studied in \cite{bica} for a  variety $X$ of  maximal Albanese dimension with  $q(X)>\dim X$, proving that  if $|2K_X|$ is not birational  then  either $X$ is birational to a theta divisor in a p.p.a.v. or there exists a fibration $f\colon X\to Y$ onto an irregular variety with $\dim Y<\dim X$.
If $X$ is a smooth theta divisor in a p.p.a.v., then  it is easy to check directly that  $|2K_X\otimes \alpha|$ is  birational for  $0\ne \alpha\in \Pic^0(X)$ and that for every $d\ge 2$ the system $|K_{X\up{d}}\otimes \alpha|$ is birational for every $\alpha \in\Pic^0(X)$.   \end{rem}

 \noindent  {\bf Warning:} in this section all the equalities  of divisors that we write down are equalities of $\Q$-divisors    up to numerical equivalence. Moreover, in accordance with the standard use  in birational geometry,  we switch to the additive notation and write $L+\alpha$ instead of $L\otimes \alpha$.  \medskip

We fix a very ample divisor $H$ on $A$ and we set $M:=a^*H$; we use the notation of diagram \ref{diag:0} and set $M_d:=a_d^*H$ for $d>1$.
 By \cite[Prop 2.3.5]{LB}, we have that $\wt{\mu_d}^*M\equiv d^2 M_d$.

 Moreover, since the question is birational,  we may assume that the map $f\colon X\to Y$  to the canonical model (which exists by \cite{bchm}) is a log resolution.

To prove Theorem \ref{thm:pluri-para} we use the   following preliminary result, that  generalizes \cite[Lem.~2.4]{ugua}:
\begin{lem} \label{lem:nef}
In the above set-up, there exists $d\gg0$ such that
$$K_{X\up{d}}=L+(n-1)M_d+E,$$ where $L$  and $E$ are $\Q$-divisors such that $L$ is  nef and big   and $E$ is  effective with normal crossings support.
\end{lem}
\begin{proof}
 By \cite[Cor.1.5]{HMc} the fibers of $f$ are rationally chain connected, hence they are contracted to points by $a$;  it follows that the Albanese map of $X$ descends to a morphism  $\ol a\colon Y\to A$.

Setting $\ol M:=\ol a^*H$, there exists $\epsilon>0$ such that  for $|t|<\epsilon$ the class $K_Y-t\ol M$ is ample and therefore $f^*K_Y-tM$ is nef and big on $X$.  By  the definition of canonical model  we have $K_X=f^*K_Y+E$, with $E$ an effective $\Q$-divisor. In addition, the support  of $E$ is a normal crossings divisor since $f$ is a log resolution. Choosing $d$ such that $\frac{n-1}{d^2}<\epsilon$ and pulling back to $X$, we have $$K_{X\up{d}}=\wt{\mu_d}^*K_X=\wt{\mu_d}^*(f^*K_Y-\frac{n-1}{d^2}M)+\frac{n-1}{d^2}\wt{\mu_d}^*M+\wt{\mu_d}^*E.$$ The statement now follows by observing that:
\begin{itemize}
\item $\frac{n-1}{d^2}\wt{\mu_d}^*M=(n-1)M_d$,
\item  $L:=\wt{\mu_d}^*(f^*K_Y-\frac{n-1}{d^2}M)=\wt{\mu_d}^*f^*K_Y-(n-1)M_d$ is the class of a nef and big line bundle,
\item $\wt{\mu_d}^*E$ is an effective $\Q$-divisor with normal crossings support.
\end{itemize}
 \end{proof}

 \begin{proof}[Proof of Thm. \ref{thm:pluri-para}]
 Since $X$ is of maximal Albanese dimension, $K_X$ is effective and $|2K_X|\subset |mK_X|$ for $m\ge 2$, hence it is enough to
 to prove the statement for $m=2$.  We may  also assume $n\ge 2$, since for $n=1$ the statement follows easily by Riemann-Roch.

By Theorem \ref{thm:factorization} and   Lemma \ref{lem:nef}  we can fix  $d\gg0$ such that:
\begin{itemize}
\item $K_{X\up{d}}=L+(n-1)M_d+E$, with $L$ nef and big and $E$  is effective with normal crossings support.
\item  for $\alpha\in \Pic^0(A)= \Pic^0(X)$ general the map given by $|2K_{X\up{d}s}+\alpha|$ coincides with $\varphi\up{d}$ (notation as in Theorem \ref{thm:factorization})
\end{itemize}

Write $E=\left\lfloor E\right\rfloor +\Delta$, so that $\Delta\ge 0$ is a $\Q$-divisor with $\left\lfloor \Delta\right\rfloor =0$ and normal crossings support.
Set $D:=K_{X\up{d}}-\left\lfloor E\right\rfloor=L+(n-1)M_d+\Delta$. We are going to show that if $C\subset X$ is the intersection of $(n-1)$ general  elements of $|M_d|$, then the map induced by $|K_{X\up{d}}+D+\alpha|$ restricts to a generically injective map on $C$ for $\alpha\in \Pic^0(A)$ general. Since $|K_{X\up{d}}+D+\alpha|\subseteq |2K_{X\up{d}}+\alpha|$ and   the map $\varphi\up{d}$ is composed with the  map $a_d\colon X\up{d}\to A$ by Theorem \ref{thm:factorization}, this will prove the statement.

First of all we prove that $|K_{X\up{d}}+D+\alpha|$ restricts to a complete system on $C$ by showing that $h^1(\mathcal I_C(K_{X\up{d}}+D+\alpha))=0$ for every $\alpha\in \Pic^0(A)$. To this end we look at  the resolution of  $\mathcal I_C$ given by the Koszul complex:
\begin{equation}\label{eq:koszul}
0\to \OO_{X\up{d}}(-(n-1)M_d)\to\dots   \to \bigwedge^2  \OO_{X\up{d}}(-M_d)^{\oplus^{n-1}}\to \OO_{X\up{d}}(-M_d)^{\oplus^{n-1}}\to \mathcal I_C\to 0
\end{equation}
Twisting \eqref{eq:koszul} by $K_{X\up{d}}+D+\alpha$ we obtain a resolution of $\mathcal I_C(K_{X\up{d}}+D+\alpha)$ such that every term is a sum of line bundles numerically equivalent to  a line bundle of the form
$$K_{X\up{d}}+D-iM_d=K_{X\up{d}}+L+(n-1-i)M_d+\Delta$$
 for some $1\le i\le n-1$. Since the hypotheses of Kawamata-Viehweg's vanishing theorem apply to each of these line bundles, it follows that $\mathcal I_C(K_{X\up{d}}+D+\alpha)$ has a resolution of length $n-2$ such that all the sheaves appearing in the resolution  have zero higher cohomology. Hence the hypercohomology spectral sequence  gives $h^1(\mathcal I_C(K_{X\up{d}}+D+\alpha))=0$ as claimed.

Finally observe that  by construction $(K_{X\up{d}}+D+\alpha)|_C$ has degree strictly greater than $2g(C)+1$ and therefore it is very ample.
 \end{proof}

\section{Examples, remarks and open questions}
In this section we give some examples of the eventual behaviour of the paracanonical system; we keep the notation of \S \ref{sec:prelim} and  \S \ref{sec:para}.
All our examples are  constructed as abelian covers; we use the notation and  the general theory of \cite{rita-abelian}.

Let $\Gamma$ be a finite abelian group,  let  $\Gamma^*:=\Hom (\Gamma, \C^*)$ be its group of characters,  and let $\pi\colon X\to Y$ be a flat $\Gamma$-cover with $Y$ smooth and $X$ normal. One has a decomposition $\pi_*\OO_X=\oplus_{\chi\in \Gamma^*}L_{\chi}\inv$, where $L_{\chi}$ is a line bundle and $\Gamma$ acts on $L_{\chi}\inv$ via the character $\chi$. In particular, we have $L_1=\OO_Y$.

Our main observation  is the following:
\begin{lem}\label{lem:examples}
 In the above setup, assume that  $X$ and $Y$ are  smooth of maximal Albanese dimension;  denote by $\ol a\colon Y\to A$ the Albanese map. If the following conditions hold:
\begin{itemize}
\item[(a)] the map $\pi$ is totally ramified, i.e., it does not factor through an  \'etale cover $Y'\to Y$ of degree $>1$;
\item[(b)]  for every $1\ne \chi\in \Gamma^*$,  $h^1(L_{\chi}\inv)=0$;
\item[(c)] there exists precisely one element $\ol \chi\in\Gamma^*$ such that $h^0_a(K_Y+L_{\ol\chi})>0$.
\end{itemize}
then:\begin{enumerate}
\item the map $a:=\ol a\circ \pi$ is the Albanese map of $X$;
\item the eventual paracanonical map $\varphi\colon X\to Z$ is composed with $\pi\colon X\to Y$;
\item  if in addition $K_Y\otimes L_{\ol\chi}$ is eventually generically birational (e.g., if $a$ is generically injective),  then $\pi$ is birationally equivalent to $\varphi$.
\end{enumerate}
\end{lem}
 \begin{proof}
We use freely  the notation of  \S \ref{sec:prelim}.

(i) First of all,  by condition (b)  we have $$q(X)=h^1(\OO_X)=\sum_{\chi}h^1(L_{\chi}\inv)=q(Y),$$  hence the induced map $\Alb(X)\to A$ is an isogeny. On the other hand, the fact that $\pi$ is  totally ramified implies that the map $\Pic^0(Y)\to \Pic^0(X)$ is injective.
 It follows   that $\Alb(X)\to A$ is  an isomorphism   and the map $a=\ol a\circ \pi$ is the Albanese map of $X$.
\smallskip

(ii) Denote by $\pi\up{d}\colon X\up{d}\to Y\up{d}$ the $\Gamma$-cover induced by $\pi\colon X\to Y$ taking base change with $Y\up{d}\to Y$. There is a cartesian  diagram:
 \begin{equation}\label{diag:1}
\begin{CD}
X\up{d}@>>>X\\
@V{\pi\up{d}}VV @VV \pi V\\
Y\up{d}@>>>Y\\
@V{\ol a_d}VV @VV \ol a V\\
A@>{\mu_d}>>A
\end{CD}
\end{equation}

 By Theorem \ref{thm:factorization}, it is enough to show that for every $d$ the system $|K_{X\up{d}}\otimes \alpha|$ is  composed with $\pi\up{d}\colon X\up{d}\to Y\up{d}$ for $\alpha\in\Pic^0(A)$ general or, equivalently, that $|K_{X\up{d}}\otimes \alpha|$ is $\Gamma$-invariant for $\alpha\in \Pic^0(A)$ general.
One has ${\pi\up{d}}_*\OO_{X\up{d}}=\oplus_{\chi\in \Gamma^*}(L_{\chi}\ud)\inv$. Hence by the formulae for abelian covers we have $h^0_a(K_{X\up{d}})=\sum_{\chi\in \Gamma^*}h^0_a(K_{Y\up{d}}\otimes L_{\chi}\ud )=h^0_a(K_{Y\up{d}}\otimes L_{\ol\chi}\ud)$, because of condition (c).
So, for $\alpha\in \Pic^0(A)$ general, we have   $H^0(K_{X\up{d}}\otimes\alpha)=H^0(K_{Y\up{d}}\otimes L_{\ol \chi}\ud\otimes \alpha)$; it follows that $\Gamma$ acts trivially on $\pp(H^0(K_{X\up{d}}\otimes\alpha))$ and therefore the map given by $|K_{X\up{d}}\otimes\alpha|$ factors through $\pi\up{d}$.
\smallskip

(iii) The proof of (ii)  shows that  eventually $\varphi\up{d}$ is birationally equivalent to the composition of $\pi\up{d}$ with the map given by $|K_{Y\up{d}}\otimes L_{\ol\chi}\ud\otimes \alpha|$ for $\alpha\in \Pic^0(A)$ general. So if $|K_Y\otimes L_{\ol\chi}|$ is eventually  generically birational, then $\varphi\up{d}$ and $\pi\up{d}$ are birationally equivalent for $d$ sufficiently large and divisible.
 \end{proof}

\begin{ex}[Examples with $m_X=2$]\label{ex:2-1}
 Every variety  of maximal Albanese dimension  $Z$      with $\chi(Z)=0$  and $\dim Z=n\ge 2$ occurs as the eventual paracanonical image for some $X$ of general type with $m_X=2$.

Let $L$ be a very ample   line bundle on $Z$,  pick  a smooth divisor $B\in |2L|$  and let $\pi\colon X\to Z$ be the double cover given by the equivalence relation $2L\equiv B$.  Since $Z$ is of maximal Albanese dimension, $K_Z$ is effective and therefore $|K_Z\otimes L|$ is birational. By Kodaira  vanishing  we have $h^1(Z, L\inv)=0$ and
$$
\chi(K_Z\otimes L)=h^0(Z, K_Z\otimes L)=h^0_a(Z, K_Z\otimes L).
$$
So by continuity the system $|K_Z\otimes L\otimes \alpha|$  is birational for general $\alpha\in\Pic^0(Z)$. By  Lemma 3.5 of \cite{articolone}, $K_{Z\up{d}}\otimes L\ud$ is generically birational for every $d$  and therefore  $\pi\colon X\to Z$ is the eventual paracanonical map  by Lemma \ref{lem:examples}.
\end{ex}

\begin{rem}
If in Example \ref{ex:2-1} we take a variety $Z$ whose  Albanese map is not generically injective, then  the eventual paracanonical map of the variety $X$ neither is an isomorphism nor coincides with the Albanese map of $X$.

Varieties $Z$ with $\chi(Z)=0$ whose Albanese map is generically finite of degree $>1$ do exist for $\dim Z\ge 2$. Observe that  it is enough to construct 2-dimensional examples, since in higher dimension one can consider the product of $Z$ with any variety of maximal Albanese dimension.  To construct an example with $\dim Z=2$ one can proceed  as follows:
\begin{itemize}
\item[--]
 take a curve $B$ with a $\Z_2$-action such that the quotient map $\pi\colon B\to B':=B/\Z_2$ is ramified and $g(B')>0$;
 \item[--] take $E$ an elliptic curve, choose a $\Z_2$-action by translation and denote by $E'$ the quotient curve $E/\Z_2$;
\item[--] set $Z:=(B\times E)/\Z_2$, where $\Z_2$ acts diagonally on the product.
\end{itemize}
The map $f\colon Z\to B'\times E'$ is a ramified double cover and it is easy to see that the Albanese map of $Z$ is the composition of $f$ with the inclusion $E'\times B'\to E'\times J(B')$. Indeed, by the universal property of the Albanese map there is a factorization $Z\to \Alb(Z)\to E'\times J(B')$. The map $\Alb(Z)\to E'\times J(B')$ is an isogeny, since $q(Z)=g(B')+1$; in addition, the dual map $E'\times J(B')\to \Pic^0(Z)$ is injective, since $f$ is ramified. We conclude that $\Alb(Z)\to E'\times J(B')$ is an isomorphism and therefore the Albanese map of $Z$ has degree 2.
\end{rem}
\begin{ex}[Examples with $m_X=4$ and $\chi(Z)=0$] \label{ex:product}
For $i=1, 2$ let $A_i$ be abelian varieties of dimension $q_i\ge 1$ and set $A=A_1\times A_2$. For $i=1,2$ let $M_i$ be a very ample line bundle on $A_i$ and choose $B_i\in |2M_i|$ general; we let $\pi\colon X\to A$ be the $\Z_2^2$-cover with branch divisors $D_1=\pr_1^*B_1$, $D_2=\pr_2^*B_2$ and $D_3=0$, given by relations $2L_1\equiv D_2$ and $2L_2\equiv D_1$, where $L_1=\eta_1\boxtimes M_2$, $L_2=M_1\boxtimes \eta_2$, with  $0\ne \eta_i\in \Pic^0( A_i)[2]$. One has $L_3=L_1\otimes L_2$, so that $L_3$ is very ample.
It is easy to check that $h^1(L_j\inv)=0$ for $j=1,2,3$ and $h^0_a(L_j)>0$ if and only if  $j=3$.
So the assumptions of Lemma \ref{lem:examples} are satisfied and arguing as in Example \ref{ex:2-1} we see that $\pi\colon X\to A$ is both the Albanese map and the eventual paracanonical map of $X$.

An interesting feature of this example is that  for  $d$  even the cover $\pi\up{d}\colon X\up{d}\to A$ is the product of two double covers $f_i\ud\colon X_i\to A_i$, $i=1,2$, branched on the pull back of $B_i$. If $q_1,q_2\ge 2$, then $q(X\up{d})=q(X)$, but if, say, $q_1=1$ then $q(X\up{d})$ grows with $d$.

In particular, if $q_1=q_2=1$ then for $d$ even the surface $X\up{d}$ is a product of bielliptic curves. \end{ex}
\begin{ex}[Examples with  large  $m_X$]\label{ex:higher}
A variation of Example \ref{ex:product} gives examples with higher values of $m_X$.
For simplicity, we describe an example of a threefold $X$ with $m_X=8$, but it is easy to see how to modify the construction to obtain example with $m_X=2^k$ for every $k$ (the dimension of the examples will also increase, of course).

For $i=1,\dots 3$, let  $A_i$ be elliptic curves, let $M_i$ be a very ample divisor on $A_i$ and let $B_i\in |2M_i|$ be a general divisor. Set $A=A_1\times A_2\times A_3$ and choose    $ \eta_1, \eta_2\in \Pic^0(A)[2]$ distinct  such that, setting $\eta_3=\eta_1+\eta_2$, then $\eta_j|_{A_i}\ne 0$ for every choice of $i,j\in\{1,2,3\}$.
Set $\Gamma:=\Z_2^3$ and denote by $\gamma_1, \gamma_2, \gamma_3$ the standard generators of $\Gamma$. We let $\pi\colon X\to A$  be  the  $\Gamma$-cover given by the following building data:
$$D_{\gamma_i}=\pr_i^*B_i, \quad i=1,\dots 3, \quad D_{\gamma}=0\quad {\rm if} \quad \gamma\ne \gamma_i,$$
and
$$L_i=\pr_i^*M_i\otimes\eta_i, \quad i=1,\dots 3.$$
It is immediate to see that the reduced fundamental relations (cf. \cite[Prop. 2.1]{rita-abelian}):
$$2L_i\equiv D_{\gamma_i}, \quad i=1,\dots 3$$
are satisfied.
The  corresponding cover $\pi\colon X\to A$ satisfies the assumptions of Lemma \ref{lem:examples} and arguing as in the previous cases  we see that $\pi\colon X\to A$ is both the Albanese map and the eventual paracanonical map of $X$.
\end{ex}
\begin{qst}
The statement of  Theorem \ref{thm:chi} shows an analogy between the eventual  behaviour of the paracanonical system of varieties of maximal Albanese dimension and the canonical map  (cf. \cite[ Thm.~3.1]{beauville-canonique}). As in the case of the canonical map, the previous examples show that  it is fairly easy to produce examples of case (b) of Theorem \ref{thm:chi} (cf. Example \ref{ex:2-1}).
On the other hand, we do not know any examples of case (a) with $m_X>1$.

For instance, in the $2$-dimensional case by Proposition \ref{prop:surface} $X$ would be a surface of general type and maximal Albanese dimension with an involution $\sigma$ such that $q(X/\sigma)=q(X)$, $p_g(X/\sigma)=p_g(X)$ and such that the eventual paracanonical map of $X/\sigma$ is  birational. If such an $X$ exists, then it has Albanese map of degree $2k$, with $k\ge2$. Indeed, the Albanese map of $X$ is composed with $\sigma$ and $K^2_{X/\sigma}\le \frac 12 K^2_X\le \frac 92 \chi(X)=\frac 92 \chi(X/\sigma)<5\chi(X/\sigma)$, where the last inequality but one  is given by the Bogomolov-Miyaoka-Yau inequality.
Theorem  6.3 (ii) of \cite{articolone}  implies that the Albanese map of $X/\sigma$ has degree $k>1$.
\end{qst}

\begin{qst}
In the surface case the eventual paracanonical degree is at most 4 by Proposition \ref{prop:surface}. In order to give a bound    for higher dimensional varieties in the same way, one would need  to bound the volume of $K_X$ in terms of $\chi(K_X)$; however Example 8.5  of \cite{articolone} shows that this is not possible for $n\ge 3$. So, it is an open question whether it is possible to give a bound on  $m_X$, for $X$ a variety of fixed dimension $n>2$. Note that Example \ref{ex:higher} shows that this bound, if it exists, has to increase with $n$.
\end{qst}

%%%%%%%%%

     \end{document}